\documentclass{amsart}

\usepackage[dvipsnames]{xcolor}
\usepackage{mathtools}

\usepackage[normalem]{ulem}

\usepackage[final]{pdfpages}

\usepackage[makeroom]{cancel}
\usepackage{graphicx}
\usepackage{cancel}
\usepackage{ulem}
\usepackage{setspace}

\usepackage{tikz}
\usetikzlibrary{arrows}
\usetikzlibrary{arrows.meta}
\usetikzlibrary{decorations.pathreplacing}

\usepackage{verbatim}

\usepackage{amsrefs}

\usepackage{graphicx}
\usepackage{mathtools}

\usepackage{graphicx}
\usepackage{float} 
\usepackage{amsfonts}
\usepackage{amsthm}
\usepackage{latexsym}
\usepackage{amsmath}
\usepackage{amssymb}
\usepackage{amscd}
\usepackage{epsf}
\usepackage{tikz}
\usetikzlibrary{matrix,arrows}
\usepackage{enumerate}
\usepackage{eepic}
\usepackage{epic}
\usepackage{amsrefs}

\theoremstyle{definition}

\newtheorem{example}{Example}

\theoremstyle{plain}

\newtheorem{proposition}{Proposition}

\newtheorem{lemma}{Lemma}

\newtheorem{method}{Method}

\usepackage[english]{babel}

\begin{document}

\title[On isospectral graphs]{On isospectral metric graphs}

\author[P.Kurasov]{Pavel Kurasov}

\author[J.Muller]{Jacob Muller}

\dedicatory{This work is dedicated to the memory of Sergey Naboko \\
-- outstanding mathematician, attentive Teacher,\\ kind friend and a great Man, who left us
too early.}

\begin{abstract}
A new class of isospectral graphs is presented. These graphs are isospectral with respect to
both the normalised Laplacian on the discrete graph and the standard differential Laplacian on the corresponding metric graph.
The new class of graphs is obtained by gluing together subgraphs with the Steklov maps possessing  special properties. 
It turns out that isospectrality is related to the degeneracy of the Steklov eigenvalues.
\end{abstract}

\maketitle

\section{Introduction}

Spectral geometry of metric graphs is a rapidly developing area of modern analysis joining together 
spectral theory of ordinary and partial differential equations,  theory of almost periodic functions, quasicrystals, analytic functions in
several variables with number theory and algebra of multivariate polynomials \cite{BeKu,KuBook,KuSa}. One of the most important directions is understanding how
topological and geometrical properties of a metric graph $ \Gamma $ (see precise definitions below) are reflected in the spectral properties of
the corresponding differential operators. The first step should be understanding these relations in the case of
the standard Laplacian -- the second order differential operator uniquely determined by the metric graph. We shall refer
to its spectrum as the spectrum of the metric graph $ \Gamma$.
Solution of the inverse spectral problem is then identical to the reconstruction of a metric
graph from its spectrum.
Note that we are interested in what is determined by the spectrum alone without taking into account additional (spectral) data like
the response operator, scattering matrix or the spectrum of any other operator associated with the same graph.

The following two results should be taken into account:
\begin{itemize}
\item  The spectrum determines the
metric graph if the edge lengths are rationally independent \cite{GuSm,KuNo}. 
\item The total length ({\it i.e.} the sum of edge lengths) and the Euler characteristic (or the number of independent cycles)
are uniquely determined by the spectrum \cite{KuJFA,KuArk}.
\end{itemize}
The first result implies that almost all graphs are uniquely determined by their spectra. Looking for 
 {\bf isospectral graphs} (sometimes called cospectral)  -- graphs
having identical spectra -- one should consider the case of rationally dependent edge lengths.
In particular one may look at the another extreme case of unilateral graphs (having all edge lengths equal to $1$).
Explicit examples of isospectal graphs
have been constructed adjusting Sunada construction \cite{Su}  originally developed for partial differential equations in domains to metric graphs \cite{GuSm,BaPa,OrBa}. In this approach
two isospectral graphs appear as subgraphs of a certain large symmetric graph. 

Discussing isospectral graphs one should always remember that the spectra of unilateral metric graphs
are determined by the spectra of the corresponding discrete normalised Laplacian matrices:
$$ L_N = \mathbb I - D^{-1/2} A D^{-1/2}, $$ 
where $ A $ is the adjacency matrix for the discrete graph $ G $ associated with $ \Gamma$ and $ D $ is the (diagonal) degree matrix for $ G$. 
Generic ({\it i.e.} not equal to $ \pi^2 m^2 , \; m \in \mathbb Z)$
eigenvalues $ \lambda_j $ of an unilateral metric graph $ \Gamma $ and the eigenvalues $ \mu_n $ of the normalised
Laplacian matrix $ L_N $ are related as \cite{Below1}
\begin{equation}
1- \cos \sqrt{\lambda_j} = \mu_n.
\end{equation}
Moreover, two unilateral metric graphs are isospectral if and only if the corresponding normalised Laplacian matrices are
isospectral and the graphs have the same number of connected components and independent cycles.
In particular isospectrality of normalised Laplacians does not necessarily imply isospectrality of
the corresponding unilateral metric graphs.

In the current article we shall look for isospectral unilateral metric graphs. It is not our goal to characterise all such graphs but rather
to understand their nature. To the best of our knowledge, the problem has not been solved neither for
normalised Laplacian matrices \cite{BuGr}, nor for usual Laplacian matrices \cite{HaSp}.
The spectra of unilateral metric graphs on up to $5$ vertices have been discussed in \cite{ChPi}, it was proven that the spectra
determine the graphs. 

We turned therefore to graphs on $6$ vertices and looked for isospectral pairs.
We found one such pair (see Section \ref{sec:ex}) and analysed the corresponding eigenfunctions
trying to find any transplantation mapping then to each other. 
Deep analysis of this pair of graphs led us
to interesting observations allowing to construct a wide family of isospectral metric graphs. Surprisingly this family gives
not known before normalised-Laplacian-isospectral discrete graphs.

Our methods are based on the analysis of Titchmarsh-Weyl $M$-functions associated with the subgraphs
and use ideas behind surgery principles in the spectral theory of metric graphs \cite{KuMaNa,KuNa,BeKeKuMu1,BeKeKuMu2}.
In contrast to the studies modifying Sunada construction derived isospectral families 
are glued together from the subgraphs having very special spectral properties.
 In some sense our approach is ideologically close to \cite{BuGr,Br},
but uses a different technique: instead of extensiv use of linear algebra we bring in methods coming from the spectral theory 
of ordinary differential operators. As a result constructed isospectral families are wider and often contain already known ones.
On the other hand it is not always possible to see the equivalence between already existing and newly proposed methods:
for example our Method \ref{Me2} reminds of Seidel switching \cite{Seidel,Br} but is not equivalent. Construction of
isospectral graphs from Cayley graphs \cite{Babai,LuSaVi} can also be generalised introducing symmetries based on $M$-functions, 
the corresponding graphs do not necessarily have any symmetries as metric spaces. 

\section{Laplacians on metric graphs and $M$-functions}

A {\bf metric graph} $ \Gamma $ is a collection of intervals 
$ E_n = [x_{2n-1}, x_{2n}] \subset \mathbb R, \; n =1 ,2, \dots, N $ -- the edges, -- joined
together at the vertices $ V_m , \; m =1,2, \dots, M $ seen as equivalence classes of the set of
all interval end points $ \{ x_j \}_{j=1}^{2N}. $ The intervals belong to different copies of $ \mathbb R$.
The end points belonging to the same vertex are identified.

The (differential) Laplace operator $ L $ is defined in the Hilbert space $ L_2 (\Gamma) = \oplus_{n=1}^N L_2 (E_n) $
on the set of functions from $ u \in  \oplus_{n=1}^N W_2^2 (E_n) $ 
satisfying at each vertex $ V_m $ standard conditions 
\begin{equation} \label{svc}
\left\{
\begin{array}{l}
x_i, x_j \in V_m \Rightarrow u(x_i) = u(x_j) ; \\[3mm]
\displaystyle \partial u (V_m) :=  \sum_{x_i \in V_m}  \partial u(x_j) = 0,
\end{array}
\right.
\end{equation}
where the oriented derivatives are defined in accordance with
$$ \partial u(x_{2n-1}) = \lim_{x \rightarrow x_{2n-1}} u'(x), \quad  \partial u(x_{2n}) = - \lim_{x \rightarrow x_{2n}} u'(x). $$
The extra sign for the right end points is required in order to treat all edges independently of their parametrisation.
The first condition in \eqref{svc}  implies that any function from the domain of the standard Laplacian is
not only continuous inside the edges (as a function from $ W_2^2$) but even on the whole graph $ \Gamma $.
The second condition, sometimes called Kirchhoff, ensures that the operator is self-adjoint.
It is not our goal to describe all possible vertex conditions (see \cite{KoSch,BeKu,KuBook}). 
Additionally we shall only use Dirichlet condition $ u(V_m) = 0 $
implying that the function is zero at all end-points joined at $ V_m $. 

On any metric graph $ \Gamma $ we choose one or several vertices to be contact vertices denoted by $ \partial \Gamma$.
Note that every inner point on an edge can be seen as a degree two vertex. For arbitrary nonreal $ \lambda, \; \Im \lambda \neq 0$
consider solutions of the differential equation
\begin{equation} \label{deq}
- u'' (x)  = \lambda u(x)
\end{equation}
satisfying standard vertex conditions at all (inner) vertices $ V_m \notin \partial \Gamma $ and continuous on $ \partial \Gamma$
(and hence on the whole $ \Gamma$). Every such solution is uniquely determines by its values on the contact set
$ u(V_m), \; V_m \in \partial \Gamma$. Otherwise the Laplacian with Dirichlet conditions on $ \partial \Gamma $ (and standard
conditions at all other vertices) would have non-real eigenvalues. We introduce the Titchmarsh-Weyl (matrix valued) $M$-function
associated with the graph $ \Gamma $ and the contact set $ \partial \Gamma$:
\begin{equation}
\mathbb M_\Gamma (\lambda): \{ u(V_m) \}_{V_m \in \partial \Gamma} \mapsto  \{ \partial u(V_m) \}_{V_m \in \partial \Gamma}
,
\end{equation}
where $  \partial u(V_m) = \sum_{x_j \in V_m} \partial u(x_j). $
$ \mathbb M_\Gamma (\lambda) $ so defined is a matrix-valued Herglotz-Nevanlinna function and
encodes information about all not-common eigenfunctions of the standard and Dirichlet Laplacians on $ \Gamma$
\cite{KuBook}.

For example the $M$-function for the interval $ \mathbf E $ of length $ 1$ with the contact set formed by both end-points
 is given by
$$ \mathbb M_{\bf E} (\lambda) =
\left(
\begin{array}{cc}
- k \cot k  & \displaystyle  \frac{k}{\sin k }  \\
 \displaystyle  \frac{k}{\sin k }  & - k \cot k  
 \end{array} \right). $$

The following two formulas express the $M$-functions through the traces of the standard (satisfying standard vertex conditions on $ \partial \mathbf \Gamma$
and Dirichlet (satisfying Dirichlet conditions on  $ \partial \mathbf \Gamma$) \cite{KuActa,KuNa,KuBook}
\begin{equation} \label{eq1PP}
 \mathbf M_{\mathbf \Gamma} (\lambda) = - \left( \sum_{n=1}^\infty \frac{\displaystyle \langle \psi_n^{\rm st} \vert_{\partial \Gamma}, \cdot \rangle_{\ell_2 (\partial \Gamma)} \psi_n^{\rm st} \vert_{\partial \Gamma}}{\lambda_n^{\rm st} - \lambda} \right)^{-1},
 \end{equation}
\begin{equation}  \label{eq2PP}
\mathbf M_{\mathbf \Gamma} (\lambda) -  \mathbf M_\Gamma (\lambda') = \sum_{n=1}^\infty \frac{\lambda - \lambda'}{(\lambda_n^D - \lambda) (\lambda_n^D - \lambda')}
\langle \partial \psi_n^D \vert_{\partial \Gamma}, \cdot \rangle_{\mathbb C^B} \partial \psi_n^D \vert_{\partial \Gamma} ,
 \end{equation}
 where $ \lambda_n^{\rm st} $ and $ \lambda_n^{\rm D} $ are the corresponding eigenvalues and $ \lambda' \in \mathbb R $ is an  arbitrary reference point not belonging to the spectra.
 The Dirichlet eigenvlalues $ \lambda_n^{/rm D}$ determine the singularities of $ \mathbb M_{\mathbf \Gamma} (\lambda) $,
 while $ \lambda_n^{\rm st} $ correspond to generalised zeroes of $ \mathbb M_{\mathbf \Gamma} (\lambda)$. 

It is clear that only the eigenfunctions with non-zero traces on the contact set contribute to the $M$-function. Such eigenfunctions are called {\bf detectable}
and the corresponding eigenvalues form {\bf detectable spectrum}. All other eigenfunctions and the corresponding spectrum are called {\bf non-detectable} or {\bf invisible}.
Non-detectable eigenfunctions satisfy both standard and Dirichlet conditions at contact vertices. Note that we treat the spectrum as a multiset so that
the same real number could be a detectable and a non-detectable eigenvalue at the same time if it is a multiple eigenvalue with some standard eigenfunctions
satisfying Dirichlet conditions on $ \partial \mathbf \Gamma$ and  some not.

The $M$-functions are Hermitian for almost all real values of $ \lambda $ as can be seen from the formulas \eqref{eq1PP} and \eqref{eq2PP}.
The eigenvalues of $ \mathbb M_{\mathbf \Gamma} (\lambda) $ will be called {\bf Steklov eigenvalues} by analogy with the Poincar\'e-Steklov problem
for elliptic partial differential equation in a domain. Note that so-defined Steklov eigenvalues depend on the spectral parameter $ \lambda$ as well as
the corresponding eigensubspaces, which we simply call {\bf Steklov subspaces}.
For regular $ \lambda $ Steklov eigenvalues and subspaces depend continuously on the spectral parameter $ \lambda$.

The graphs with equal $M$-functions will be called Steklov-equivalent. Such graphs are also known as  {\it isoscattering}, {\it isopolar}, or {\it isophasal} graphs \cite{BaSaSm2010}
as attaching infinite edges to contact points leads to the scattering matrix, which is essentially a Cayley transform of the $M$-function. 

The Steklov eigenvalues between the singular points depend continuously on the spectral parameter $ \lambda$ and are given by non-decreasing functions.
The detectable eigenvalues of the standard Laplacian on $ \mathbf \Gamma $ are obtained when the Steklov eigenvalues cross the zero line.
Note that these points may coincide with the singularities of $ \mathbb M(\lambda) $, therefore one speaks about generalised zeroes -- the singularities of the
inverse function $ - \mathbb M^{-1} (\lambda)$. Note that Steklov-equivalent graphs are not necessarily isospectral as their non-detectable eigenvalues may be different,

\section{Two isospectral graphs from the complete graph $ K_5$.} \label{sec:ex}

Searching for isospectral graphs we decided to inspect equilateral graphs on few vertices. The graphs on up to $5$ vertices have been
considered in \cite{ChPi}. We looked at graphs on $6$ vertices aiming to find isospectral pairs. There are altogether more than $100$ such graphs
and it was a tough job to analyse their spectra. We present here just the result of these tedious calculations.

The spectra of unilateral metric graphs are best described by zeroes of secular polynomials \cite{BaGa,KuSa,KuBook} determined by
\begin{equation}
P_G (z) = \det \Big( \mathbf E (z) - \mathbf S_{\rm v} \Big),
\end{equation}
where
$$ \mathbf E(z) = {\rm diag}\, \left( \begin{array}{cc} 0 & z \\
z & 0
\end{array} \right) $$
and $ \mathbf S_{\rm v} $ is the vertex scattering matrix. The spectrum of the Laplacian on the
unilateral graph $ \Gamma $ is obtained by putting $ z = e^{ik} $ as zeros of the trigonometric
polynomial
\begin{equation}
p_\Gamma (k) = P_G (e^{ik}).
\end{equation}
Only zeros of the secular polynomials are relevant, hence these polynomials should be treated
projectively so that two polynomials different by a multiplicative factor are equal.

For example the secular polynomial for the complete graph $ \mathbf{K}_5 $ is
\begin{equation}
P_{K_5} = (z-1 )^7 (z+1 )^5 (2z^2 + z + 2 )^4.
\end{equation}
Order of the zeroes  coincides with  the multiplicity of the corresponding non-zero eigenvalues.
For example $ z_1= 1 $ is a zero of multiplicity $7$ implying that $ \lambda = (2 \pi m)^2 $ is
an eigenvalue of multiplicity $ 7 = 1 + \beta_1(\mathbf{K}_5) $,
where $ \beta_1 $ denotes the number of independent cycles in the graphs  (the first Betti number).
The multiplicity of $ \lambda = 0 $ is just equal to $1$ as the graph is connected.

 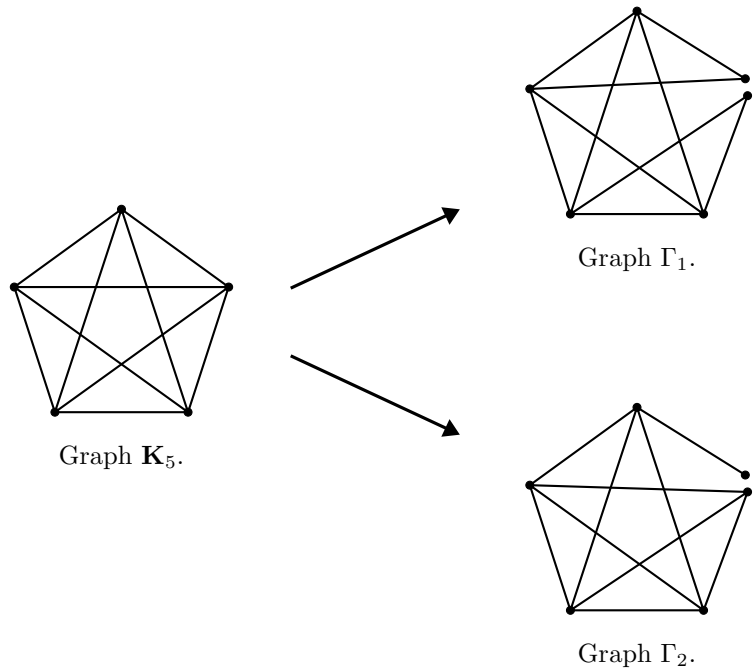
\begin{figure}[h]

\begin{tikzpicture}[scale=1.5]

\draw  [thick]  (-0.59,-0.8) to (0.59,-0.8);
\draw  [thick] (0.59,-0.8) to (0.95,0.31);
\draw  [thick] (0.95,0.31) to (0,1);
\draw  [thick] (0,1) to (-0.95,0.31);
\draw  [thick] (-0.95,0.31) to (-0.59,-0.8);

\draw  [thick]  (-0.59,-0.8) to (0.95,0.31);
\draw  [thick]  (-0.59,-0.8) to (0,1);

\draw  [thick] (0.59,-0.8) to (0,1);
\draw  [thick] (0.59,-0.8) to (-0.95,0.31);

\draw  [thick] (0.95,0.31) to (-0.95,0.31);

\draw[fill] (-0.59,-0.8) circle (1pt);       
\draw[fill] (0.59,-0.8)circle (1pt);    
\draw[fill] (0.95,0.31) circle (1pt);       
\draw[fill] (-0.95,0.31)circle (1pt);    
\draw[fill] (0,1) circle (1pt);     

\node at (0,-1.2) {Graph $ \mathbf K_5$.};

\draw[-Triangle, very thick](1.5, 0.3) -- (3, 1.);

\draw[-Triangle, very thick](1.5, -0.3) -- (3, -1.);

\begin{scope} [xshift=130, yshift=50]
\draw  [thick]  (-0.59,-0.8) to (0.59,-0.8);
\draw  [thick] (0.59,-0.8) to (0.98,0.25);
\draw  [thick] (0.96,0.4) to (0,1);
\draw  [thick] (0,1) to (-0.95,0.31);
\draw  [thick] (-0.95,0.31) to (-0.59,-0.8);

\draw  [thick]  (-0.59,-0.8) to (0.98,0.25);
\draw  [thick]  (-0.59,-0.8) to (0,1);

\draw  [thick] (0.59,-0.8) to (0,1);
\draw  [thick] (0.59,-0.8) to (-0.95,0.31);

\draw  [thick] (0.96,0.4) to (-0.95,0.31);

\draw[fill] (-0.59,-0.8) circle (1pt);       
\draw[fill] (0.59,-0.8)circle (1pt);    
\draw[fill] (0.96,0.4)circle (1pt);    
\draw[fill] (0.98,0.25) circle (1pt);       
\draw[fill] (-0.95,0.31)circle (1pt);    
\draw[fill] (0,1) circle (1pt);     
 
\node at (0,-1.2) {Graph $ \Gamma_1$.};

\end{scope}

       \begin{scope} [xshift=130, yshift=-50]
\draw  [thick]  (-0.59,-0.8) to (0.59,-0.8);
\draw  [thick] (0.59,-0.8) to (0.98,0.25);
\draw  [thick] (0.96,0.4) to (0,1);
\draw  [thick] (0,1) to (-0.95,0.31);
\draw  [thick] (-0.95,0.31) to (-0.59,-0.8);

\draw  [thick]  (-0.59,-0.8) to (0.98,0.25);
\draw  [thick]  (-0.59,-0.8) to (0,1);

\draw  [thick] (0.59,-0.8) to (0,1);
\draw  [thick] (0.59,-0.8) to (-0.95,0.31);

\draw  [thick] (0.98,0.25) to (-0.95,0.31);


\draw[fill] (-0.59,-0.8) circle (1pt);       
\draw[fill] (0.59,-0.8)circle (1pt);    
\draw[fill] (0.96,0.4)circle (1pt);    
\draw[fill] (0.98,0.25) circle (1pt);       
\draw[fill] (-0.95,0.31)circle (1pt);    
\draw[fill] (0,1) circle (1pt);     
 
\node at (0,-1.2) {Graph $ \Gamma_2$.};

\end{scope}

\end{tikzpicture}
 \caption{{Two isospectral graphs via $  \mathbf{K}_5$.}}
\label{FigK5}
\end{figure}

Let us consider the unilateral graphs $ \Gamma_1 $ and $ \Gamma_2$ obtained from $ \mathbf{K}_5 $ by
chopping one of the vertices into two. The new vertices have degrees $ 2, 2 $ and $ 1,3$ respectively.
The corresponding secular polynomials coincide,
\begin{equation}
P_{\Gamma_1} (z) = P_{\Gamma_2} (z) = (z-1 )^6 (z+1 )^4 (2z^2 + z + 2 )^3 (2z^4 + z^3 + 2 z^2 + z + 2),
\end{equation}
implying that the two connected  graphs are isospectral.

It would be natural to derive 
 isospectrality of the pair $ (\Gamma_1, \Gamma_2) $ using
symmetries of $ \mathbf K_5$; we provide an alternative explanation looking
at the $ M$-functions and analysing the corresponding
Steklov subspaces. In fact analysing the above explicit example we arrived at several
methods allowing construction of non-trivial isospectral graphs. These methods are described
in the following sections.

\section{ 
Understanding isospectrality looking at  Steklov subspaces} \label{sec:swap}

In this section we provide an explicit explanation why the graphs $ \Gamma_1$ and $ \Gamma_2 $ are
isospectral looking at the Steklov subspaces of the two subgraphs that differ them.
In some sense instead of analysing how the complete graph $ \mathbf K_5 $ is chopped
to get the pair, we shall analyse how to obtain the isospectral pair by extending
the graph $ \mathbf K_4$.

Both graphs $ \Gamma_1 $ and $ \Gamma_2 $ contain the complete graph $ \mathbf K_4 $.
Hence each of these graphs can be obtained by attaching a certain equilateral graph $ \mathbf{Q}_j $ to $ \mathbf K_4$.
Let us amend Figure \ref{FigK5}, so that our construction will become more transparent.

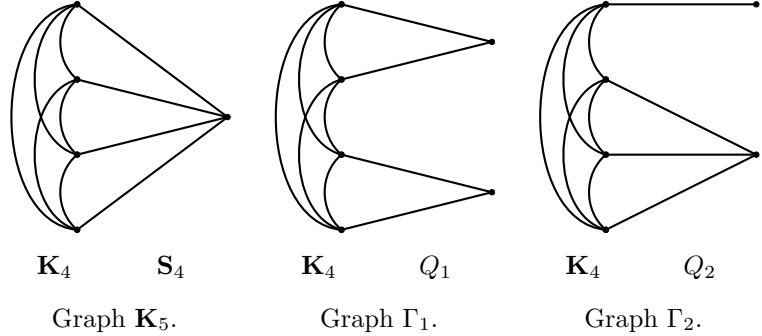
\begin{figure}[h]

\begin{tikzpicture}[scale=1]

\begin{scope} [xshift=-100]

\coordinate (a) at (0,0);
\coordinate (b) at (0,1);
\coordinate (c) at (0,2);
\coordinate (d) at (0,3);

\draw[fill] (a) circle (1pt);     
\draw[fill] (b) circle (1pt);    
\draw[fill] (c) circle (1pt);    
\draw[fill] (d) circle (1pt);      

\draw[thick,bend left=+50] (a) edge (b);
\draw[thick,bend left=+50] (b) edge (c);
\draw[thick,bend left=+50] (c) edge (d);

\draw[thick,bend left=+75] (a) edge (c);
\draw[thick,bend left=+75] (b) edge (d);

\draw[thick,bend left=+85] (a) edge (d);

\coordinate (e) at (0.,0);
\coordinate (f) at (0.,1);
\coordinate (g) at (0.,2);
\coordinate (h) at (0.,3);
\coordinate (i) at (2,1.5);
\coordinate (j) at (2,3);

\draw[thick] (e) edge (i);
\draw[thick] (f) edge (i);
\draw[thick] (g) edge (i);

\draw[thick] (h) edge (i);

\draw[fill] (e) circle (1pt);      
\draw[fill] (f) circle (1pt);      
\draw[fill] (g) circle (1pt);      
\draw[fill] (h) circle (1pt);      
\draw[fill] (i) circle (1pt);      


\node at (-0.3,-.5) {$ \mathbf K_4$};
\node at (1.25,-.5) {$ \mathbf S_4$};

\node at (0.5,-1.2)  {Graph $ \mathbf K_5$.};

\end{scope}

\coordinate (a) at (0,0);
\coordinate (b) at (0,1);
\coordinate (c) at (0,2);
\coordinate (d) at (0,3);

\draw[fill] (a) circle (1pt);     
\draw[fill] (b) circle (1pt);    
\draw[fill] (c) circle (1pt);    
\draw[fill] (d) circle (1pt);      

\draw[thick,bend left=+50] (a) edge (b);
\draw[thick,bend left=+50] (b) edge (c);
\draw[thick,bend left=+50] (c) edge (d);

\draw[thick,bend left=+75] (a) edge (c);
\draw[thick,bend left=+75] (b) edge (d);

\draw[thick,bend left=+85] (a) edge (d);

\coordinate (e) at (0.,0);
\coordinate (f) at (0.,1);
\coordinate (g) at (0.,2);
\coordinate (h) at (0.,3);
\coordinate (i) at (2,0.5);
\coordinate (j) at (2,2.5);

\draw[thick] (e) edge (i);
\draw[thick] (f) edge (i);

\draw[thick] (g) edge (j);
\draw[thick] (h) edge (j);

\draw[fill] (e) circle (1pt);      
\draw[fill] (f) circle (1pt);      
\draw[fill] (g) circle (1pt);      
\draw[fill] (h) circle (1pt);      
\draw[fill] (i) circle (1pt);      
\draw[fill] (j) circle (1pt);


\node at (-0.3,-.5) {$ \mathbf K_4$};
\node at (1.25,-.5) {$ Q_1$};

\node at (0.5,-1.2) {Graph $ \Gamma_1$.};

\begin{scope} [xshift=100]

\coordinate (a) at (0,0);
\coordinate (b) at (0,1);
\coordinate (c) at (0,2);
\coordinate (d) at (0,3);

\draw[fill] (a) circle (1pt);     
\draw[fill] (b) circle (1pt);    
\draw[fill] (c) circle (1pt);    
\draw[fill] (d) circle (1pt);      

\draw[thick,bend left=+50] (a) edge (b);
\draw[thick,bend left=+50] (b) edge (c);
\draw[thick,bend left=+50] (c) edge (d);

\draw[thick,bend left=+75] (a) edge (c);
\draw[thick,bend left=+75] (b) edge (d);

\draw[thick,bend left=+85] (a) edge (d);

\coordinate (e) at (0.,0);
\coordinate (f) at (0.,1);
\coordinate (g) at (0.,2);
\coordinate (h) at (0.,3);
\coordinate (i) at (2,1);
\coordinate (j) at (2,3);

\draw[thick] (e) edge (i);
\draw[thick] (f) edge (i);
\draw[thick] (g) edge (i);

\draw[thick] (h) edge (j);

\draw[fill] (e) circle (1pt);      
\draw[fill] (f) circle (1pt);      
\draw[fill] (g) circle (1pt);      
\draw[fill] (h) circle (1pt);      
\draw[fill] (i) circle (1pt);      
\draw[fill] (j) circle (1pt);


\node at (-0.3,-.5) {$ \mathbf K_4$};
\node at (1.25,-.5) {$ Q_2$};

\node at (0.5,-1.2)  {Graph $ \Gamma_2$.};

\end{scope}

\end{tikzpicture}
 \caption{{Two isospectral graphs extending $ \mathbf K_4$.}}
\label{FigK4}
\end{figure}

The spectra of these graphs can be calculated using their $M$-functions
associated with the $4$ vertices belonging to the $ \mathbf K_4$ subgraph
on the left.

 All three graphs in Fig. \ref{FigK4} are divided by contact vertices into two or three separate graphs:
 $ \mathbf K_4 $ and star graphs $ \mathbf S_d$ with $ d=1,2,3,4$ edges. The $M$-functions for the
 composed graphs are equal to the sums of $M$-functions of the subgraphs.
 
 Let us calculate the $M$-functions for the subgraphs appearing in the decomposition. Each time we shall
 use that the $M$-functions commute with the permutations of the end points implying that the Steklov
 subspaces are invariant under permutations and can only be equal to
 $$ \mathcal L \mathbf 1 \quad \mbox{and} \quad ( \mathcal L \mathbf 1 )^\perp , $$
 where $ \mathbf 1 $ denotes the vector with all coordinates equal to $ 1 $ and the size
 depending on the number of contact vertices in the subgraph. Hence, there are just two
 Steklov eigenvalues for each subgraph.
 
 \subsection*{\bf Complete graph $ \mathbf K_4$.}
The  $M$-function associated with all vertices in $ \mathbf K_4 $ is equal to the sum of $M$-functions
for the edges:
 \begin{equation}
 \mathbb M_{\mathbf K_4} (\lambda) =
 \left(
 \begin{array}{cccc}
 - 3 k \cot k  & \displaystyle  \frac{k}{\sin k } & \displaystyle  \frac{k}{\sin k } & \displaystyle  \frac{k}{\sin k } \\[2mm]
  \displaystyle  \frac{k}{\sin k } & - 3 k \cot k  &  \displaystyle  \frac{k}{\sin k } & \displaystyle  \frac{k}{\sin k } \\[2mm]
    \displaystyle  \frac{k}{\sin k } & \displaystyle  \frac{k}{\sin k } &  - 3 k \cot k  & \displaystyle  \frac{k}{\sin k } \\[2mm]
   \displaystyle  \frac{k}{\sin k } & \displaystyle  \frac{k}{\sin k } & \displaystyle  \frac{k}{\sin k } &  - 3 k \cot k   \\
 \end{array}
 \right).
 \end{equation}
 Of course one needs to pay attention to which contact points are joined by each edge.
 The Steklov eigenvalues  denoted by $ \mu_{1,2}^{\mathbf K_4} 
 $ are
 $$ \mu_1 ^{\mathbf K_4} =  - 3k \cot k + 3 \frac{k}{\sin k} =  3 k \tan k/2, \quad \mu_2 ^{\mathbf K_4} =  - 3k \cot k - \frac{k}{\sin k} . $$
 Denoting by $ \mathbf P_{\mathbf 1} $ and $ \mathbf P_{\mathbf 1^\perp} $ the orthogonal projectors on the Steklov subspaces, 
 the $M$-function can be written as
 \begin{equation} \label{MK}
 \mathbb M_{\mathbf K_4} (\lambda) = \mu_1^{\mathbf K_4} (\lambda)  \mathbf P_{\mathbf 1} + 
 \mu_2^{\mathbf K_4} (\lambda) \mathbf P_{\mathbf 1^\perp} .
 \end{equation}
 
  \subsection*{\bf Star graph $ \mathbf S_d$.}
 The eigenvalues of the $M$-function for the star graph $ \mathbf S_d $ associated with all degree one vertices 
 do not depend
 on the valency $ d$
 \begin{equation}
 \mu_1^{\mathbf S} (\lambda) =  k \tan k,\quad  \mu_2^{\mathbf S}(\lambda)  = - k \cot k.
\end{equation}
The corresponding $M$-function also commutes with the permutations and therefore
 can be written similar to \eqref{MK}:
  \begin{equation} \label{MS}
 \mathbb M_{\mathbf S_d} (\lambda) = \mu_1^{\mathbf S_d} (\lambda) \mathbf P_{\mathbf 1} + 
 \mu_2^{\mathbf S_d} (\lambda)  \mathbf P_{\mathbf 1^\perp} .
 \end{equation}
The dependence on $ d$ is through the dimension of the vector $ \mathbf 1$.

 The spectrum of $ \mathbf K_5$ associated with the eigenfunctions not identically equal to zero
 at the contact vertices is given by the equation
 $$ \det \Big(  \mathbb M_{\mathbf K_4} (\lambda) +  \mathbb M_{\mathbf S_4} (\lambda)  \Big) = 0, $$
  which, taking into account \eqref{MK} and \eqref{MS}, reduces to
  \begin{equation}
  \left\{
  \begin{array}{l}
\displaystyle  \mu_1^{\mathbf K_4} + \mu_1^{\mathbf S} =  3 k \tan k/2 +  k \tan k = 0, \\[3mm]
\displaystyle  \mu_2^{\mathbf K_4} + \mu_2^{\mathbf S} = - 4k \cot k - \frac{k}{\sin k}  = 0.
\end{array} \right.
  \end{equation}
  We get scalar equations as the eigensubspaces for $  \mathbb M_{\mathbf K_4} (\lambda) $ and  $
   \mathbb M_{\mathbf S_4} (\lambda) $ coincide.
  The first equation determines simple eigenvalues, while the second one leads to triple ones.
  Additional eigenvalues are due to eigenfunctions equal to zero at all contact points. Such
  eigenvalues are of the form  $ (2 \pi m)^2, m =1,2, \dots$ and have multiplicity $ \beta_1 (\mathbf K_5) = 6$.

Let us turn to the graphs $ \Gamma_j $. The $M$-functions for $ \mathbf Q_1 $ and $ \mathbf Q_2 $ are expressed using
the $M$-functions of $1,2,$ and $3$-star graphs:
$$
\begin{array}{ccl}
\mathbb M_{\mathbf Q_1} (\lambda) & = & \displaystyle \mu_1^{\mathbf S} \mathbf P_{(1,1,0,0)} 
+  \mu_2^{\mathbf S} \mathbf P_{(1,-1,0,0)} 
+  \mu_1^{\mathbf S} \mathbf P_{(0,0,1,1)} 
+  \mu_2^{\mathbf S} \mathbf P_{(0,0,1,-1)} \\
& = & \mu_1^{\mathbf S} \Big( \mathbf P_{(1,1,1,1)} +  \mathbf P_{(1,1,-1,-1)} \Big)
+  \mu_2^{\mathbf S} \Big( \mathbf P_{(1,-1,0,0)}  + \mathbf P_{(0,0,1,-1)} \Big), \\[3mm]
\mathbb M_{\mathbf Q_2} (\lambda) & = & \displaystyle \mu_1^{\mathbf S} \mathbf P_{(1,0,0,0)} 
+  \mu_1^{\mathbf S} \mathbf P_{(0,1,1,1)} 
+  \mu_2^{\mathbf S} \mathbf P_{(0,2,-1,-1)} 
+  \mu_2^{\mathbf S} \mathbf P_{(0,0,1,-1)} \\
& = & \mu_1^{\mathbf S} \Big( \mathbf P_{(1,1,1,1)} +  \mathbf P_{(3,-1,-1,-1)} \Big)
+  \mu_2^{\mathbf S} \Big( \mathbf P_{(0,2,-1,-1)}  + \mathbf P_{(0,0,1,-1)} \Big).
\end{array}
$$
The main difference to the previous example is that both eigenvalues have multiplicity $2$.
The corresponding Steklov eigensubspaces  are different, also  
 the vectors $ \mathbf 1 = (1,1,1,1) $ and  $ (0,0,1,-1)$ are common for $ \mathbf Q_1 $ and $ \mathbf Q_2$.
The two $M$-functions can be obtained from each other by swapping the Steklov subspaces.

The detectable spectrum is given by
$$ \det \Big( \mathbb M^{\mathbf K_4} (\lambda) + \mathbb M^{\mathbf Q_j} (\lambda) \Big) = 0 . $$
The eigensubspaces corresponding to the four (non-distinct) eigenvalues of $  \mathbb M_{\mathbf K_4} (\lambda) $
and each $  \mathbb M_{\mathbf Q_j} (\lambda) $ can be chosen equal, also these subspaces depend on $ j = 1,2$.
Therefore the above equation reduces to three scalar equations
\begin{equation}
\left\{
\begin{array}{l}
\displaystyle  \mu_1^{\mathbf K_4} + \mu_1^{\mathbf S} =  3 k \tan k/2 +  k \tan k = 0, \\[3mm]
\displaystyle  \mu_2^{\mathbf K_4} + \mu_1^{\mathbf S} = - 3k \cot k - \frac{k}{\sin k} + k \tan k = 0, \\[3mm]
\displaystyle  \mu_2^{\mathbf K_4} + \mu_2^{\mathbf S} = - 4k \cot k - \frac{k}{\sin k}  = 0.
\end{array} \right.
  \end{equation}
  The first two equations determine simple eigenvalues while the third one - double. The equations are identical for
  the two graphs, but the corresponding subspaces are different:
  \begin{equation}
  \begin{array}{ll} 
  \mathbf \Gamma_1: &  \mathcal L \{(1,1,1,1)\}, \\
 & \mathcal L \{(1,1,-1,-1)\}, \\
 & \mathcal L \{(1,-1,0,0),(0,0,1,-1)\};
    \end{array} \quad 
    \begin{array}{ll} 
  \mathbf \Gamma_2: & \mathcal L \{(1,1,1,1)\}, \\
 & \mathcal L \{(3,-1,-1,-1)\}, \\
 & \mathcal L \{(0,2,-1,-1),(0,0,1,-1)\}.
    \end{array} 
    \end{equation}
    The invisible spectrum is formed by the
    eigenvalues $ (2 \pi m)^2, \; m= 1,2, \dots $ with the multiplicity equal to the
    genus of the two graphs $ \beta_1 (\Gamma_1) = \beta_1 (\Gamma_2) = 5$.
    
    Both $M$-functions for $ \Gamma_1 $ and $ \Gamma_2 $ are obtained from the $M$-function for $ \mathbf S_4 $ by
    changing the eigenvalue on a certain one-dimensional eigensubspace from $ \mu_2^{\mathbf S} $
    to $ \mu_1^{\mathbf S}. $ Note that the one-dimensional subspaces are different for $ \Gamma_1 $ and
    $ \Gamma_2$ but this does not affect the spectrum of the composed graph. 
    So the mechanism  behind isospectrality of this pair is as follows: 
    the $M$-function for the original graph ($\mathbf K_5$ in our example) is given by a sum
    of the $M$-functions for two subgraphs sharing the same 
    Steklov eigensubspace, one having dimension at least $2$, 
    the $ M$-functions for the isospectral pair are obtained by amending the eigenvalue
    on a certain one-dimensional subspace of the distinguished Steklov eigensubspace.
    The key point is that such procedure gives $M$-functions associated with metric graphs
    obtained by chopping the original graph.

    \begin{example}
    The graph $ \mathbf K_4 $  can be substituted with the graph $ \mathbf S_4$, since the only property of $ \mathbf K_4$
    which was used is that
    the $M$-function commutes with the rotations. This gives us two isospectral graphs $ \Gamma_1'$ and $ \Gamma_2'$ 
    presented in Fig. \ref{Fig8}.
    \begin{figure}[h]

\begin{tikzpicture}[scale=1]

\coordinate (a) at (0,0);
\coordinate (b) at (0,1);
\coordinate (c) at (0,2);
\coordinate (d) at (0,3);
\coordinate (aa) at (-2,1.5);

\draw[fill] (a) circle (1pt);     
\draw[fill] (b) circle (1pt);    
\draw[fill] (c) circle (1pt);    
\draw[fill] (d) circle (1pt);   
\draw[fill] (aa) circle (1pt);      

\draw[thick] (a) edge (aa);
\draw[thick] (b) edge (aa);
\draw[thick] (c) edge (aa);

\draw[thick] (d) edge (aa);

\coordinate (e) at (0.,0);
\coordinate (f) at (0.,1);
\coordinate (g) at (0.,2);
\coordinate (h) at (0.,3);
\coordinate (i) at (2,0.5);
\coordinate (j) at (2,2.5);

\draw[thick] (e) edge (i);
\draw[thick] (f) edge (i);

\draw[thick] (g) edge (j);
\draw[thick] (h) edge (j);

\draw[fill] (e) circle (1pt);      
\draw[fill] (f) circle (1pt);      
\draw[fill] (g) circle (1pt);      
\draw[fill] (h) circle (1pt);      
\draw[fill] (i) circle (1pt);      
\draw[fill] (j) circle (1pt);


\node at (-0.3,-.5) {$ \mathbf S_4$};
\node at (1.25,-.5) {$ \mathbf Q_1$};

\node at (0.5,-1.2) {Graph $ \Gamma_1'$.};

\begin{scope} [xshift=130]

\coordinate (a) at (0,0);
\coordinate (b) at (0,1);
\coordinate (c) at (0,2);
\coordinate (d) at (0,3);
\coordinate (aa) at (-2,1.5);

\draw[fill] (a) circle (1pt);     
\draw[fill] (b) circle (1pt);    
\draw[fill] (c) circle (1pt);    
\draw[fill] (d) circle (1pt);   
\draw[fill] (aa) circle (1pt);      

\draw[thick] (a) edge (aa);
\draw[thick] (b) edge (aa);
\draw[thick] (c) edge (aa);

\draw[thick] (d) edge (aa);

\coordinate (e) at (0.,0);
\coordinate (f) at (0.,1);
\coordinate (g) at (0.,2);
\coordinate (h) at (0.,3);
\coordinate (i) at (2,1);
\coordinate (j) at (2,3);

\draw[thick] (e) edge (i);
\draw[thick] (f) edge (i);
\draw[thick] (g) edge (i);

\draw[thick] (h) edge (j);

\draw[fill] (e) circle (1pt);      
\draw[fill] (f) circle (1pt);      
\draw[fill] (g) circle (1pt);      
\draw[fill] (h) circle (1pt);      
\draw[fill] (i) circle (1pt);      
\draw[fill] (j) circle (1pt);


\node at (-0.3,-.5) {$ \mathbf S_4$};
\node at (1.25,-.5) {$ \mathbf Q_2$};

\node at (0.5,-1.2)  {Graph $ \Gamma_2'$.};

\end{scope}

\end{tikzpicture}
 \caption{{Two isospectral graphs extending $ \mathbf S_4$.}}
\label{Fig8}
\end{figure}
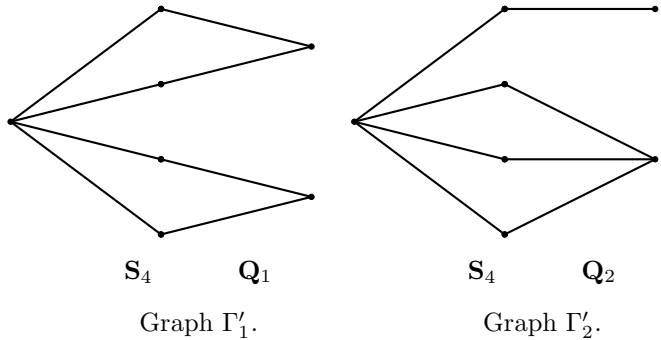
One may draw these graphs in a slightly better way as presented in Fig. \ref{Fig82},
indicating that this is probably the simplest pair of isospectral connected  metric graphs.

    \begin{figure}[h]

\begin{tikzpicture}[scale=2]

\coordinate (a) at (-0.5,0);
\coordinate (b) at (0,0);
\coordinate (c) at (0.5,0);

\draw[thick] (a)  circle (0.5);
\draw[thick] (c)  circle (0.5);

\draw[fill] (b) circle (1pt);   
\draw[fill] (1,0) circle (1pt);  
\draw[fill] (-1,0) circle (1pt);  

\node at (0.,-.7) {Figure-eight graph.};

\begin{scope} [xshift=75]

\coordinate (a) at (-1.,0);
\coordinate (b) at (0,0);
\coordinate (c) at (1,0);

\draw[thick] (a)  edge (c);

\draw[thick,bend left=+100] (b) edge (c);
\draw[thick,bend left=-100] (b) edge (c);

\draw[fill] (a) circle (1pt); 
\draw[fill] (b) circle (1pt); 
\draw[fill] (c) circle (1pt); 

\node at (0.,-0.7) {Watermelon on a stick graph.};

\end{scope}

\end{tikzpicture}
 \caption{The simplest pair of isospectral metric graphs.}
\label{Fig82}
\end{figure}
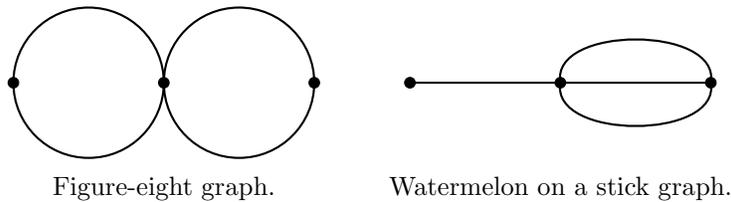

    \end{example}
    
    In what follows we formulate several principles which can be used to construct isospectral graphs.
    
\section{Method  $1$: extending isospectral graphs}

The method below requires that one isospectral pair is already known, but allows one to
get an infinite family of isospectral pairs by extension.

\begin{method}[Extending of isospectral pairs]  \label{Meth1}
Assume that $ \mathbf R_1 $ and $ \mathbf R_2 $ are two  Stek\-lov-equivalent graphs (with identical $M$-functions), which in addition are isospectral.
Then gluing these graphs to an arbitrary graph $ \mathbf K $ yields a new pair of isospectral graphs.
\end{method}

With this observation in mind we shall always try to identify common part in any isospectral pair and
check whether there exists a simpler isospectral pair behind.

We prove first one Lemma which will be used several times below.

\begin{lemma} \label{Le1}
Assume that $ M$-functions $ \mathbf M_{\mathbf \Gamma_j} (\lambda) $ for two metric graphs $ \mathbf \Gamma_1 $ and $ \mathbf \Gamma_2 $  
are unitary equivalent
\begin{equation}
\mathbb M_{\Gamma_1} (\lambda)  \stackrel{\rm unitary}{\sim}  \mathbb M_{\Gamma_2} (\lambda)
.
\end{equation}
 Then the graphs are isospectral if and only if 
the graphs have the same non-detectable spectra.
\end{lemma}
\begin{proof}
It is implicit from the statement of the Lemma  that the contact sets for $ \Gamma_1 $ and $ \Gamma_2 $ have the same dimension.
Unitary equivalence of the $M$-functions 
together with the explicit formula \eqref{eq1PP} imply that the detectable spectra
of $ \mathbf \Gamma_1 $ and $ \mathbf \Gamma_2 $ coincide. Then isospectrality of the graphs imply is equivalent to the
coincidence of the
non-detectable spectra. Remember that the spectrum is treated as a multiset.
\end{proof}

\begin{proof}[Proof of Method \ref{Meth1}]
We denote by $ \mathbf \Gamma_j , \; j = 1,2,$ the graphs ofiained by gluing $ \mathbf  K$ and $  \mathbf R_j$.
To prove that the graphs $  \mathbf  \Gamma_j $  are isospectral it is convenient to
 divide the eigenfunctions and hence the spectra $ \Sigma (\Gamma) $ into two sets:
 \begin{itemize}
 \item Localised eigenfunctions -- the eigenfunctions supported by the subgraphs $  \mathbf R_j $;
 \item Delocalised eigenfunctions -- the eigenfunctions having support not localised to $  \mathbf R_j $.
 \end{itemize}
 If the eigenvalue   is simple, then the eigenfunction is unique and it is clear whether it is localised or not.
 In the case of multiple eigenvalues $ \lambda_j $ one should examine whether the corresponding  eigenfunctions
 can be chosen localised to 
 $  \mathbf R_j $ or not. Then the multiplicity $ \dim_{\rm loc} (\lambda_j)$  of the localised eigenvalue is the maximal number of linearly independent
 localised eigenfunctions and coincides with the multiplicity of non-detectable eigenvalues on $ \mathbf R_j$. The complementing dimension determines the nonlocalised multiplicity $ \dim_{\rm n/loc} (\lambda_j) $
 $$ \dim_{\rm n/loc} (\lambda_j)  = \dim (\lambda_j) - \dim_{\rm loc} (\lambda_j) . $$

 The invisible eigenfunctions on $  \mathbf R_1 $ and $  \mathbf R_2 $ have the same eigenvalues including multiplicities by Lemma \ref{Le1}.
  It follows that
the localised spectra of $  \mathbf \Gamma_1 $ and $  \mathbf \Gamma_2$ coincide as the corresponding eigenfunctions
 are just invisible eigenfunctions on $  \mathbf R_1 $ and $ \mathbf R_2$ respectively, extended by zero to the rest of $ \mathbf \Gamma_j$. 
 
 Let us prove that the spectra corresponding to delocalised eigenfunctions coincide as well.  Consider any such eigenfunction $ \psi_{\lambda_j} $
 on $  \mathbf \Gamma_1 $. Then an eigenfunction $ \psi'_{\lambda_j} $ on $  \mathbf \Gamma_2 $ can be constructed having the same values on $  \mathbf K $:
 $$ \psi_{\lambda_j} \vert_{\mathbf K} =  \psi'_{\lambda_j} \vert_{\mathbf K}. $$
 The function is extended to $  \mathbf R_2 \subset \Gamma_2 $ having the same values on $ \partial  \mathbf R_2 $ 
 and therefore having the same
 derivatives on $ \partial  \mathbf R_2 $ as $  \mathbf R_1 $ and $ \mathbf R_2 $ have identical $M$-functions. 
 It follows that $  \psi'_{\lambda_j} $ constructed in this way
 satisfies standard vertex conditions everywhere on $  \mathbf \Gamma_2$. It is clear that this map is injective and the roles of $ \mathbf \Gamma_1 $ and $  \mathbf \Gamma_2 $
 can be exchanged, hence even multiplicities of the eigenvalues coincide.
\end{proof}

 \section{Method $2$: exchanging Steklov-equivalent subgraphs}
 
 \begin{method}[Exchanging Steklov-equivalent subgraphs] \label{Me2}
 Assume that a metric graph $  \mathbf \Gamma $ contains two subgraphs $  \mathbf R_1 $ and $  \mathbf R_2$ with equal
 $M$-functions:
 \begin{equation}
 \mathbb M_{ \mathbf R_1} (\lambda) = \mathbb M_{ \mathbf R_2} (\lambda).
 \end{equation}
 Then exchanging the subgraphs $  \mathbf R_1 $ and $  \mathbf R_2$ one obtains a graph $  \mathbf \Gamma'$ isospectral to
 $  \mathbf \Gamma$.
 \end{method}
 Note that we do not require that the subgraphs $  \mathbf R_1 $ and $  \mathbf R_2 $ are isospectral.
 
 \begin{proof}
 As before we divide the eigenfunctions into localised and delocalised ones.

 The localised eigenfunctions $ \psi_{\lambda_j} $ and $ \psi_{\lambda_j}' $ on $  \mathbf \Gamma $ and $ \mathbf  \Gamma'$ respectively, are 
 essentially the same -- they are mapped in accordance to:
 \begin{equation}
 \psi_{\lambda_j}' \vert_{\mathbf R_i} =  \psi_{\lambda_j}\vert_{\mathbf R_i} .
 \end{equation}
 Note that in the formula above we identified the points belonging to the same subgraphs $  \mathbf R_i $
 belonging to $  \mathbf \Gamma $ and $  \mathbf \Gamma'$. The supports of localised eigenfunctions move together
 with the subgraphs $  \mathbf  R_j$.
 
 The delocalised eigenfunctions are identical outside the subgraphs
  \begin{equation}
 \psi_{\lambda_j}' \vert_{\mathbf \Gamma'\setminus( \mathbf R_1 \cap \mathbf R_2)} =  \psi_{\lambda_j}\vert_{\mathbf \Gamma \setminus(\mathbf  R_1 \cap\mathbf  R_2)} ,
 \end{equation}
 where we identified points in $ \mathbf \Gamma $ and $ \mathbf \Gamma' $ outside the subgraphs.
 The function $ \psi_{\lambda_j}' $ is continued inside the subgraphs $ \mathbf R_j $ having the same boundary
 values at the contact points $ \partial \mathbf R_{1,2} $ but exchanging the subgraphs:
 \begin{equation} 
 \psi_{\lambda_j}' \vert_{\partial \mathbf R_1} =  \psi_{\lambda_j} \vert_{\partial \mathbf R_2}, \quad \psi_{\lambda_j}' \vert_{\partial \mathbf R_2} =  \psi_{\lambda_j} \vert_{\partial \mathbf R_1}.
 \end{equation}
 For each of the subgraphs there is always a solution of \eqref{deq} having the above values at the contact points.
 Even if the $M$-function is singular the values at contact points satisfy the conditions ensuring solvability 
 of the problem.
 As the subgraphs have equal $M$-functions,  we have the same relation for the derivatives 
 \begin{equation} 
 \partial \psi_{\lambda_j}' \vert_{\partial \mathbf R_1} =   \partial  \psi_{\lambda_j} \vert_{\partial \mathbf R_2}, \quad  \partial \psi_{\lambda_j}' \vert_{\partial \mathbf R_2} =   \partial  \psi_{\lambda_j} \vert_{\partial \mathbf R_1},
 \end{equation}
 and the extended function is an eigenfunction on $ \Gamma' $.

 The mappings of the localised and non-localised eigenfunctions are clearly injective and the roles of $ \mathbf \Gamma $ and $ \mathbf \Gamma'$
 can be exchanged, hence these graphs are  isospectral.
 \end{proof}
 
 It is clear that instead of a pair of subgarphs one may use any number of such subgraphs
 and permute them arbitrarily leading to families of isospectral graphs.
  
 In the rest of this section we shall discuss how  graphs with equal $ M$-functions may be obtained. It is not our
 goal to characterise all pairs of Steklov-equivalent graphs.
 
 \subsection*{Steklov-equivalent graphs}
 Probably the simplest pair of Steklov-equivalent graphs can be obtained by modifying our example of isospectral 
 graphs (see Fig. \ref{Fig83}).
      \begin{figure}[h]

\begin{tikzpicture}[scale=2]

\coordinate (a) at (-0.5,0);
\coordinate (b) at (0,0);
\coordinate (c) at (0.5,0);

\draw[thick] (0,0) .. controls (1,1.) and (2,0.2) .. (0,0);
\draw[thick] (0,0) .. controls (1,-1.) and (2,-0.2) .. (0,0);


\draw[fill] (b) circle (1pt);   
\draw[fill] (1.15,0.4) circle (1pt);   
\draw[fill] (1.15,-0.4) circle (1pt);   

\node at (-0.2,0) {$ \partial R_1$};

\node at (0.,-.7) {Graph $R_1$.};

\begin{scope} [xshift=75]

\coordinate (a) at (1.,0.5);
\coordinate (b) at (0,0);
\coordinate (c) at (1,-0.5);

\draw[thick] (a)  edge (b);
\draw[thick] (c)  edge (b);

\draw[thick,bend left=+25] (b) edge (c);
\draw[thick,bend left=-25] (b) edge (c);

\draw[fill] (a) circle (1pt); 
\draw[fill] (b) circle (1pt); 
\draw[fill] (c) circle (1pt); 

\node at (-0.2,0) {$ \partial R_2$};

\node at (0.,-0.7) {Graph $R_2$.};

\end{scope}

\end{tikzpicture}
 \caption{Steklov-equivalent metric graphs.}
\label{Fig83}
\end{figure}
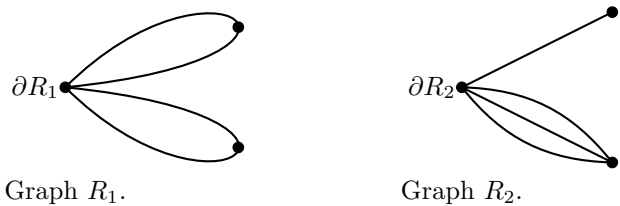
One can find a lot of examples of Steklov-equivalent graphs in the literature where they are often
called  isoscattering graphs (see \cite{BaSaSm2010}).
Note that most of these examples come from isospectral graphs implying that the graphs
are not only Steklov-equivalent, but also share the same non-detectable spectrum.

 \subsection*{Steklov-equivalent graphs via inner symmetries} 
 Following \cite{BoKu} we consider the case where the graph $\mathbf  R$ has a 
 non-trivial (inner) symmetry $g$ 
 not affecting the contact points:
 \begin{equation}
 \label{defg}
  g \mathbf R = \mathbf R, \quad  g \neq \mathbb I_{\mathbf R}, \quad g \vert_{\partial \mathbf R} = \mathbb I_{\partial \mathbf R},
  \end{equation}
 where $ \mathbb I $ denotes the identity transformations.
 Since the graph is finite, there exists an integer $ n \geq 2 $ such that $ g^n = \mathbb I_{\mathbf R}$.
 Consider any solution $ u $ to the differential equation \eqref{deq} which is continuous on $ \mathbf R $ and 
 satisfies standard conditions at non-contact vertices $ V_m \notin \partial \mathbf  R$.
It is clear that $ gu, g^2 u, \dots,  g^{n-1} u $ are also solutions.
We introduce 
 $$ u_m (x) : = \sum_{i=0}^n e^{\frac{2\pi}{n} im} g^{i} u (x) . $$
 It is clear that $ u_0 (x) $ is invariant under $ g $
 $$ g u_0(x) = u(x) $$
 while all other functions $ u_m (x), m = 1,2, \dots, n-1 $ are quasi-invariant:
 $$ g u_m (x) = e^{- \frac{2 \pi}{n} m } u_m (x). $$
 Then $ g \vert_{\partial R} = \mathbb I_{\partial R} $ implies that
 $$ u_m \vert_{\partial R} \equiv 0, \quad m =1,2, \dots, n-1 .$$
 Therefore calculating $ \mathbb M_{\mathbf R} (\lambda) $ only the symmetric solution $ u_0 (x) $ may be
 used.
 
 Consider any point $ x_0 \in \mathbf R $ not invariant under $ g$ and denote by $ x_j $ its images:
 $$ x_j := g^j x_0, \quad j =1,2, \dots, n-1. $$
 The function $ u_0 $ attains the same values at all $ x_j, j =0,1,2,\dots, n-1$, hence 
 joining these points together into a vertex $ V_0 $ will not affect the symmetric solutions:
 the function $ u_0 $ is not only continuous at $ V_0 $ but the sum of oriented derivatives
 is zero. It does not matter whether the original point $ x_0 $ was an inner point on an edge,
 or a vertex. Let us denote the graph obtained from $ \mathbf R $ by introducing the new vertex $ V_0 =
 \{ x_0, x_1, \dots, x_{n-1} \} $ by $ \mathbf R'$. Note that the new graph is invariant under the same symmetry $g$
(properly  understood).  It is clear that the graphs $ \mathbf R $ and $ \mathbf R' $ have the same
 $M$-functions but their non-detectable spectra need not be equal as the quasi-invariant eigenfunctions may 
 fail to satisfy standard conditions at the new vertex $ V_0 $.
 
 We illustrate this construction with one explicit example appeared already in \cite{KuSt}.
 We take the cycle graph with two opposite points as contact vertices. The symmetry $ g $
 is then the transformation exchanging the two edges keeping the vertices. Joining together
 any two symmetric points inside the edges turns the graph into figure-eight graph
 with the most remote points as contact vertices. It is clear that the spectrum of
 the second graph depends on the position of the point $ x_0 $ on the edges.
 Hence these two graphs are Steklov-equivalent but not isospectral.

       \begin{figure}[h]

\begin{tikzpicture}[scale=2]

\coordinate (a) at (-1,0);
\coordinate (b) at (1,0);

\draw[thick,bend left=+50] (a) edge (b);
\draw[thick,bend left=-50] (a) edge (b);

\draw[fill] (b) circle (1pt);   
\draw[fill] (a) circle (1pt);   

\draw[fill] (.5,0.35) circle (0.5pt);   
\draw[fill] (.5,-0.35) circle (0.5pt);   
\node at (.5,0.5) {$x_0$};
\node at (.5,-0.5) {$x_1$};

\draw[-stealth] (.5,0.35) -- (0.5,0.05);
\draw[-stealth] (.5,-0.35) -- (0.5,-0.05);

\node at (-1.2,0) {$ \partial R$};
\node at (1.2,0) {$ \partial R$};

\node at (0.,-.7) {Graph $R$.};

\begin{scope} [xshift=100]

\coordinate (a) at (-1.,0.);
\coordinate (b) at (0.5,0);
\coordinate (c) at (1,0);

\draw[thick,bend left=+50] (b) edge (a);
\draw[thick,bend left=-50] (b) edge (a);

\draw[thick,bend left=+50] (b) edge (c);
\draw[thick,bend left=-50] (b) edge (c);

\draw[fill] (a) circle (1pt); 
\draw[fill] (b) circle (1pt); 
\draw[fill] (c) circle (1pt); 

\node at (-1.2,0) {$ \partial R'$};
\node at (1.2,0) {$ \partial R'$};

\node at (0.5,0.2) {$V_0$};

\node at (0.,-0.7) {Graph $R'$.};

\end{scope}

\end{tikzpicture}
 \caption{Steklov-equivalent graphs via inner symmetries.}
\label{Fig84}
\end{figure}
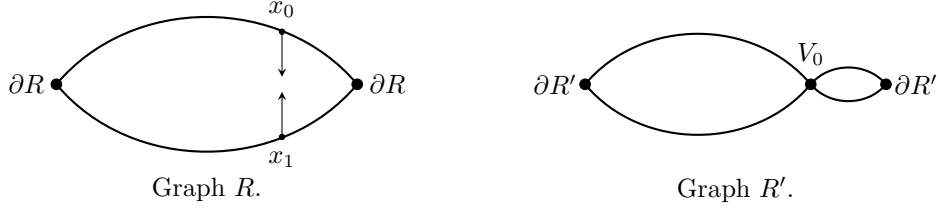

\section{Methods $3$ and $4$: swapping Steklov subspaces} \label{sec:meth3}

The third method is inspired by the example described in Section \ref{sec:swap}. We tried to generalise the idea behind.

\begin{method}[Swapping Steklov subspaces $1$] \label{Me3}
Let $ \mathbf K $ be a metric graph with a certain degenerate Steklov eigenvalue $ \mu(\lambda) $ with the eigensubspace $ V(\lambda), \; \dim V(\lambda) > 1$.
Let $ \mathbf Q_j $ be two isospectral metric graphs such that
\begin{equation}
\begin{array}{c}
\mathbb M_{\mathbf Q_1} (\lambda) \vert_{V(\lambda)} \stackrel{\rm unitary}{\sim}  \mathbb M_{\mathbf Q_2} (\lambda) \vert_{V(\lambda)} , \\[3mm]
\mathbb M_{\mathbf Q_1} (\lambda) \vert_{V(\lambda)^\perp} = \mathbb M_{\mathbf Q_2} (\lambda) \vert_{V(\lambda)^\perp} ,
\end{array} 
\end{equation}
holds, then the graphs $ \mathbf  \Gamma_j, \; j =1,2 $ obtained by gluing together $ \mathbf K $ and $ \mathbf Q_j $ are isospectral.
\end{method}
\begin{proof} Proving that $ \mathbf \Gamma_1 $ and $ \mathbf \Gamma_2 $ are isospectral we shall again use localised and delocalised
eigenfunctions. 
Our assumptions imply that the graphs $ \mathbf Q_j $ have the same Steklov eigenvalues (also the graphs do not need to be
Steklov-equivalent), in other words, the corresponding $M$-functions are
unitary-equivalent:
\begin{equation}
\mathbb M_{\mathbf Q_1} (\lambda) = \mathbb U (\lambda) \mathbb M_{\mathbf Q_2} (\lambda) \mathbb U^{-1} (\lambda),
\end{equation}
with 
$$ \mathbb U (\lambda)\vert_{V(\lambda)^\perp} = \mathbb I_{V(\lambda)^\perp}.$$
Hence it holds
$$ \mathbb M_{\mathbf K} (\lambda) = \mathbb U (\lambda) \mathbb M_{\mathbf K} (\lambda) \mathbb U^{-1} (\lambda) $$
as $ \mathbb M_{\mathbf K} (\lambda) \vert_{V(\lambda)} = \mu(\lambda) \mathbb I_{V(\lambda)}. $
It follows that Steklov eigenvalues for $ \mathbf \Gamma_1 $ and $ \mathbf \Gamma_2$ coincide between the singular points:
\begin{equation}
\underbrace{\mathbb M_{\mathbf K} (\lambda) + \mathbb M_{\mathbf Q_1} (\lambda)}_{\mathbb M_{\mathbf \Gamma_1} (\lambda)}  = 
 \mathbb U (\lambda)  \Big( \underbrace{\mathbb M_{\mathbf K} (\lambda) + \mathbb M_{\mathbf Q_1} (\lambda)}_{\mathbb M_{\mathbf \Gamma_2} (\lambda)} \Big)  \mathbb U^{-1} (\lambda),
\end{equation}
which in turn imply that the detectable spectra of $ \mathbf \Gamma_1 $ and $ \mathbf \Gamma_2 $ coincide.

Lemma \ref{Le1} implies that non-detectable spectra coincide as well, since the subgraphs $ \mathbf Q_j $ are isospectral and have the same Steklov spectra and the subgraph $ \mathbf K $
is common.
\end{proof}

Unfortunately we do not have any recipe to obtain a wide class of graphs satisfying the assumptions. Our main tool are numerous
generalisations of the example from Section \ref{sec:swap}, where the Steklov subspaces are independent of the spectral parameter $ \lambda$.
These examples are present in the next section.

It is also possible to generalise the last method by relaxing the assumption that the graphs $ \mathbf \Gamma_j$ have a common part (the graph $ \mathbf K$).

\begin{method}[Swapping Steklov subspaces $2$]
Let $ \big( \mathbf K_1,  \mathbf K_2 \big)$ and $ \big(\mathbf Q_1,  \mathbf Q_2 \big)$ be two pairs of graphs with the same number of contact points and
unitary equivalent $ M$-functions, such that
for  a certain invariant for $ \mathbb M_{\mathbf K_j} (\lambda) $ and $ \mathbb M_{\mathbf Q_j} (\lambda) $ subspace $ V(\lambda) $ it holds
\begin{equation}
\mathbb M_{\mathbf K_1} (\lambda) \vert_{V(\lambda)^\perp} =  \mathbb M_{\mathbf K_2} (\lambda) \vert_{V(\lambda)^\perp}, \quad \mathbb M_{\mathbf Q_1} (\lambda) \vert_{V(\lambda)^\perp}  =
\mathbb M_{\mathbf Q_2} (\lambda) \vert_{V(\lambda)^\perp} .
\end{equation}
Denoting by $ \mu_s ^{\mathbf K_j}(\lambda), \;  \mu_s ^{\mathbf Q_j}(\lambda), \; s = 1,2,, \dots , S$ the pairwise equal Steklov eigenvalues, let us assume
that the corresponding eigensubspaces can be chosen equal:
$$ V_s^{\mathbf K_j} (\lambda) =  V_s^{\mathbf Q_j} (\lambda) , \; s = 1,2, \dots, S,  \quad j =1,2. $$
If in addition the graphs
$$ \mathbf K_1 \cup \mathbf Q_1 \quad \mbox{and} \quad  \mathbf K_2 \cup \mathbf Q_2  $$
are isospectral, then the graphs $ \mathbf \Gamma_1$ and $ \mathbf \Gamma_2 $ obtained by gluing together $ \mathbf K_j $ and $ \mathbf Q_j $
are also isospectral
\end{method}

The proof of this method goes along the same lines.

\section{Clarifying example}

In this section we construct an example of two isospectral metric graphs using most of the formulated methods. We hope that this
example is sufficiently transparent so that it clearly illustrates all possibilities. Good examples sometimes tell more than
complete characterisation of all possibilities.

We start by constructing the graph $ \mathbf K $ with degenerate Steklov eigenvalues. This can be done using symmetries. For example let us
take the complete graph $ \mathbf K_5$ with the vertices $ V_1, \dots, V_5$, 
 where all edges are substituted with arbitrary Steklov-equivalent graphs $ \mathbf A $ and $ \mathbf  B $
having two contact points 
$$ \mathbf M^{\mathbf A} (\lambda) =  \mathbf M^{\mathbf B} (\lambda).$$
One may use for example the graphs presented in Fig. \ref{Fig84}. 
The $M$-function associated with the contact set $ \partial \mathbf K = \{ V_j \}_{j=1}^5$
commutes with the permutations, although the graph does not necessarily have any symmetry after the edges are substituted with
the graphs $ \mathbf A $ and $ \mathbf B$ (see Fig. \ref{FigEx}, the edges are marked by the red and magenta colours). To the constructed graph one may attach the star graph on $ 5 $ edges,
where again every edge is substituted with the same graphs $ \mathbf C $  having two contact points (these graphs are marked by the cyan colour). The degeneracy of the Steklov
eigenvalues persists, although their values are affected. Finally to the central vertex $ V_6 $  in the star graph
we attach arbitrary graph $ \mathbf E$ (marked by green colour in the figure). The constructed graph $ \mathbf K $ has a complicated
structure, may have no symmetries but the corresponding $M$-function associated with $ \partial \mathbf K $
for almost every $ \lambda$
has degenerate
eigenvalues of multiplicity $4$ with the corresponding eigensubspace given by $ V(\lambda) =  \mathbf 1^\perp$. 

Constructing the graph $ \mathbf K $ we could also attach Steklov-equivalent graphs to the vertices $ V_1, \dots, V_5$, or combine the two ideas. It is not our aim to
describe all possibilities here.

 \begin{figure}[h]

\begin{tikzpicture}[scale=2.]

\coordinate (a) at (-0.3,-0.05);
\coordinate (b) at (0.72,0.72);
\coordinate (c) at (-0.1,1);
\coordinate (d) at (0.75,-0.25);
\coordinate (e) at (0.25,-1);

\coordinate (f) at (-1.5,0.5);

\draw  [line width=0.5mm,magenta] (a) -- (b) node [midway, above, sloped] (TextNode) {\color{magenta} \tiny \bf A};
\draw  [line width=0.5mm,magenta] (b) -- (c) node [midway, above, sloped] (TextNode) {\color{magenta} \tiny \bf A};
\draw  [line width=0.5mm,red] (c) -- (d) node [midway, above, sloped] (TextNode) {\hspace{2mm} \color{red} \tiny \bf B};
\draw  [line width=0.5mm,red] (d) -- (e) node [midway, below, sloped] (TextNode) {\color{red} \tiny \bf B};
\draw  [line width=0.5mm,red] (e) -- (a) node [midway, above, sloped] (TextNode) {\color{red} \tiny \bf B};

\draw  [line width=0.5mm,red] (a) -- (c) node [midway, above, sloped] (TextNode) {\color{red} \tiny \bf B};
\draw  [line width=0.5mm,red] (a) -- (d) node [midway, below, sloped] (TextNode) {\color{red} \tiny \bf B};

\draw  [line width=0.5mm,magenta] (b) -- (d) node [midway, above, sloped] (TextNode) {\color{magenta} \tiny \bf A};
\draw  [line width=0.5mm,magenta] (b) -- (e) node [midway, above, sloped] (TextNode) {\hspace{2mm} \color{magenta} \tiny \bf A};

\draw  [line width=0.5mm,red] (c) -- (e) node [midway, above, sloped] (TextNode) {\hspace{-4mm} \color{red} \tiny \bf B};

\draw  [line width=0.5mm,cyan] (f) -- (a) node [midway, below, sloped] (TextNode) {\hspace{4mm} \color{cyan} \tiny \bf C};
\draw  [line width=0.5mm,cyan] (f) -- (b) node [midway, above, sloped] (TextNode) {\color{cyan} \tiny \bf C};
\draw  [line width=0.5mm,cyan] (f) -- (c) node [midway, above, sloped] (TextNode) {\color{cyan} \tiny \bf C};
\draw  [line width=0.5mm,cyan] (f) -- (d) node [midway, above, sloped] (TextNode) {\color{cyan} \tiny \bf C};
\draw  [line width=0.5mm,cyan] (f) -- (e) node [midway, above, sloped] (TextNode) {\color{cyan} \tiny \bf C};

\draw[line width=2pt,green,fill=green,opacity=0.4] (f) .. controls (-2.5,0) and (-2.5,1)  .. (f);
\node at (-2.,0.5) {\tiny  \bf E };

\draw[fill] (a) circle (1pt);       
\draw[fill] (b)circle (1pt);    
\draw[fill] (c) circle (1pt);       
\draw[fill] (d)circle (1pt);    
\draw[fill] (e) circle (1pt); 
\draw[fill] (f) circle (1pt);     

\node at (-0.1,1.1) {\tiny \bf 1};
\node at (0.74,0.8) {\tiny \bf 2};

\node at (0.85,-0.2) {\tiny \bf 3};
\node at (0.25,-1.1) {\tiny \bf 4};
\node at (-0.33,-0.15) {\tiny \bf 5};

\node at (-1.5,0.4) {\tiny \bf 6};

\node at (0,-1.5) {Graph $ \mathbf K$.};

\draw[-Triangle, line width=1.mm,dashed](1.3, 1.8) -- (0.8, 1.3);
\draw[-Triangle, line width=1.mm,dotted](1.3, -1.1) -- (0.8, -0.6);

\begin{scope} [xshift=50, yshift=50]

\coordinate (a1) at (-0.3,-0.05);
\coordinate (b1) at (0.72,0.72);
\coordinate (c1) at (-0.1,1);
\coordinate (d1) at (0.75,-0.25);
\coordinate (e1) at (0.25,-1);

\coordinate (f1) at (1.5,0.8);
\coordinate (g1) at (1.5,-0.5);

\draw  [line width=0.5mm,black] (c1) -- (f1) node [midway, above, sloped] (TextNode) {\color{black} \tiny \bf D};
\draw  [line width=0.5mm,black] (b1) -- (f1) node [midway, below, sloped] (TextNode) {\color{black} \tiny \bf D};
\draw  [line width=0.5mm,black] (a1) -- (g1) node [midway, below, sloped] (TextNode) {\hspace{2mm} \color{black} \tiny \bf D};
\draw  [line width=0.5mm,black] (d1) -- (g1) node [midway, above, sloped] (TextNode) {\color{black} \tiny \bf D};
\draw  [line width=0.5mm,black] (e1) -- (g1) node [midway, above, sloped] (TextNode) {\color{black} \tiny \bf D};

\draw[line width=2pt,blue,fill=blue,opacity=0.4] (f1) .. controls (2.5,1) and (2.5,1.5)  .. (f1);
\draw[line width=2pt,blue,fill=blue,opacity=0.4] (f1) .. controls (2.5,0.6) and (2.5,0.1)  .. (f1);

\draw[line width=2pt,blue,fill=blue,opacity=0.4] (g1) .. controls (2.5,-0.3) and (2.5,0.2)  .. (g1);
\draw[line width=2pt,blue,fill=blue,opacity=0.4] (g1) .. controls (2.5,-0.7) and (2.5,-1.2)  .. (g1);
\draw[line width=2pt,blue,fill=blue,opacity=0.4] (g1) .. controls (2.7,-0.3) and (2.7,-0.7)  .. (g1);

\node at (1.9,0.98) {\tiny  \bf F };
\node at (1.9,0.63) {\tiny  \bf F };

\node at (1.9,-0.31) {\tiny  \bf F };
\node at (2.0,-0.5) {\tiny  \bf F };
\node at (1.9,-0.69) {\tiny  \bf F };

\draw[fill] (a) circle (1pt);       
\draw[fill] (b)circle (1pt);    
\draw[fill] (c) circle (1pt);       
\draw[fill] (d)circle (1pt);    
\draw[fill] (e) circle (1pt); 
\draw[fill] (f) circle (1pt);     
\draw[fill] (g) circle (1pt);

\node at (-0.1,1.1) {\tiny \bf 1};
\node at (0.72,0.6) {\tiny \bf 2};

\node at (0.75,-0.15) {\tiny \bf 3};
\node at (0.25,-1.1) {\tiny \bf 4};
\node at (-0.3,-0.15) {\tiny \bf 5};

\node at (1.5,0.7) {\tiny \bf 6};
\node at (1.5,-0.6) {\tiny \bf 7};

\node at (1,-1.2) {Graph $ \mathbf Q_1$.};

\end{scope}

\begin{scope} [xshift=50, yshift=-50]

\coordinate (a2) at (-0.3,-0.05);
\coordinate (b2) at (0.72,0.72);
\coordinate (c2) at (-0.1,1);
\coordinate (d2) at (0.75,-0.25);
\coordinate (e2) at (0.25,-1);

\coordinate (f2) at (1.5,1);
\coordinate (g2) at (1.5,-0.5);

\draw  [line width=0.5mm,black] (c2) -- (f2) node [midway, above, sloped] (TextNode) {\color{black} \tiny \bf D};
\draw  [line width=0.5mm,black] (b2) -- (g2) node [midway, above, sloped] (TextNode) {\color{black} \tiny \bf D};
\draw  [line width=0.5mm,black] (a2) -- (g2) node [midway, below, sloped] (TextNode) {\hspace{2mm} \color{black} \tiny \bf D};
\draw  [line width=0.5mm,black] (d2) -- (g2) node [midway, above, sloped] (TextNode) {\color{black} \tiny \bf D};
\draw  [line width=0.5mm,black] (e2) -- (g2) node [midway, above, sloped] (TextNode) {\color{black} \tiny \bf D};

\draw[line width=2pt,blue,fill=blue,opacity=0.4] (f2) .. controls (2.7,0.8) and (2.7,1.2)  .. (f2);

\draw[line width=2pt,blue,fill=blue,opacity=0.4] (g2) .. controls (2.5,-0.2) and (2.5,0.3)  .. (g2);
\draw[line width=2pt,blue,fill=blue,opacity=0.4] (g2) .. controls (2.5,-0.8) and (2.5,-1.3)  .. (g2);
\draw[line width=2pt,blue,fill=blue,opacity=0.4] (g2) .. controls (2.7,-0.5) and (2.7,-0.9)  .. (g2);
\draw[line width=2pt,blue,fill=blue,opacity=0.4] (g2) .. controls (2.7,-0.5) and (2.7,-0.1)  .. (g2);

\node at (2,1) {\tiny  \bf F };

\node at (1.9,-0.29) {\tiny  \bf F };
\node at (2.0,-0.42) {\tiny  \bf F };
\node at (2.0,-0.58) {\tiny  \bf F };
\node at (1.9,-0.71) {\tiny  \bf F };

\draw[fill] (a) circle (1pt);       
\draw[fill] (b)circle (1pt);    
\draw[fill] (c) circle (1pt);       
\draw[fill] (d)circle (1pt);    
\draw[fill] (e) circle (1pt); 
\draw[fill] (f) circle (1pt);     
\draw[fill] (g) circle (1pt);

\node at (-0.1,1.1) {\tiny \bf 1};
\node at (0.72,0.6) {\tiny \bf 2};

\node at (0.75,-0.15) {\tiny \bf 3};
\node at (0.25,-1.1) {\tiny \bf 4};
\node at (-0.3,-0.15) {\tiny \bf 5};

\node at (1.5,0.9) {\tiny \bf 6};
\node at (1.5,-0.6) {\tiny \bf 7};

\node at (1,-1.2) {Graph $ \mathbf Q_2$.};

\end{scope}

\draw[dashed] (a) to (a1);
\draw[dashed] (b) to (b1);
\draw[dashed] (c) to (c1);
\draw[dashed] (d) to (d1);
\draw[dashed] (e) to (e1);

\draw[thick,dotted] (a) to (a2);
\draw[thick,dotted] (b) to (b2);
\draw[thick,dotted] (c) to (c2);
\draw[thick,dotted] (d) to (d2);
\draw[thick,dotted] (e) to (e2);

\end{tikzpicture}
 \caption{{Isospectral graphs by gluing.\newline
 The numbers indicate positions of different vertices.}}
\label{FigEx}
\end{figure}
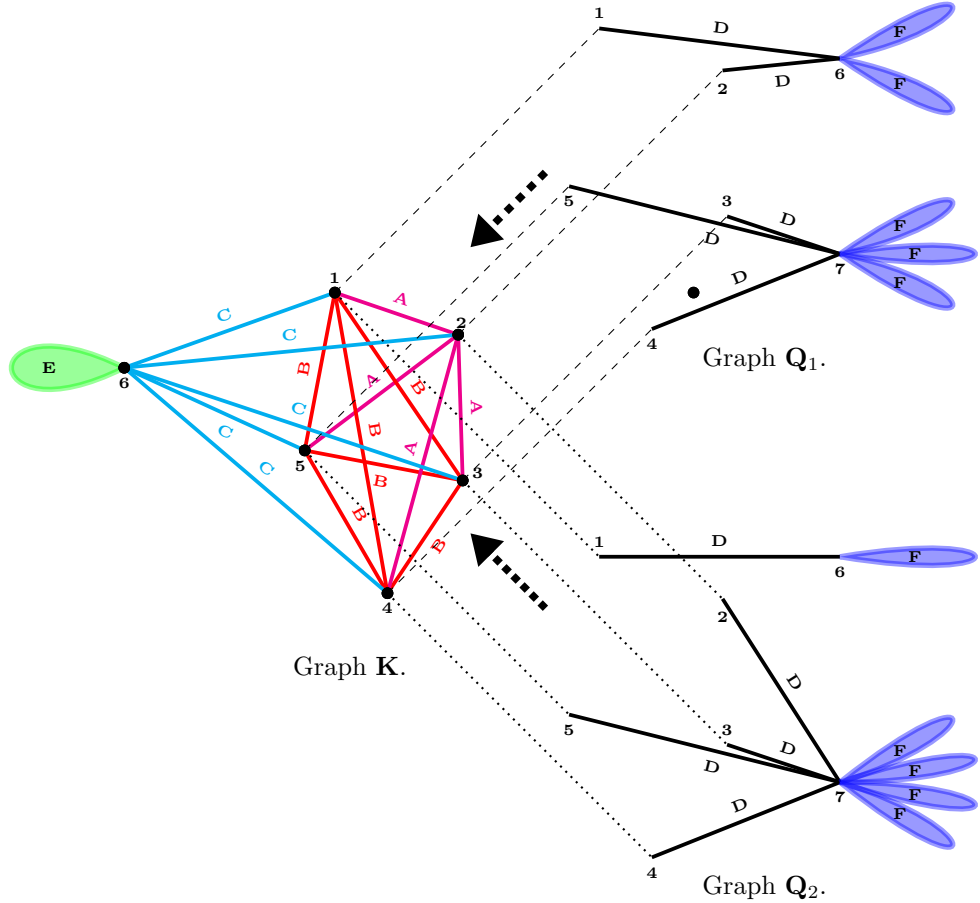

To the graph $ \mathbf K $ we shall add two different graphs $ \mathbf Q_j $ constructed by spitting the star
graph on five edges. Consider the $5$-star graph where degree $1$ vertices are denoted by $ V_1, \dots, V_5$
and with all edges substituted with the graphs $ \mathbf D $
having two contact points (marked by the back colour in the figure). 
Then each of the  graphs $ \mathbf Q_j $ is formed by splitting the central vertex in the star graph
into the vertices $ V_6 $ and $V_7 $ having valencies either $2 $ and $3$, or $ 1$ and $4$. 
To each of these vertices $ V_j , \; j =6,7,$ we attach as many copies of an arbitrary graph $ \mathbf F $
(marked by the blue colour in the figure) as the number of graphs $ D $ joined at the vertex.

The contact set for $ \mathbf Q_j $ is formed by the first five vertices $ \partial \mathbf Q_j = \{ V_j\}_{j=1}^5$,
but these vertices are naturally divided into two sets corresponding to the two connected components
in each of the graphs.
For each connected component
 the  vector $ \mathbf 1 $ (having dimension $ 1,2,3,$ or $4$) is not only a Steklov eigenvector (for almost any $ \lambda$) but
the corresponding Steklov eigenvalue does not depend on the dimension. To achieve that different number of graphs $ \mathbf F $
were attached.
The orthogonal complement is
again a Steklov eigensubspace with the eigenvalue equal to the $M$-function for the graph $ \mathbf D $
with Dirichlet condition introduced at the opposite contact point. 

Gluing pairwise the contact vertices in $ \mathbf K $ and $ \mathbf Q_j$, $ j=1,2,$ we get two isospectral graphs $ \Gamma_1 $ and $\Gamma_2$.
These graphs may have no symmetries. One may even take the graphs $ \mathbf A, \mathbf B, \mathbf C ,$ and $ \mathbf D $
Steklov-equivalent but having different  invisible eigenvalues. Then it is possible to exchange
these graphs in arbitrary way destroying all remains of possibly existing symmetries. The same argument
applies to graphs $ \mathbf E $
and $ \mathbf F$: for example one may take $5$ Steklov-equivalent graphs $ \mathbf F_j, \j=1, \dots 5$ and attach them in
arbitrary manner to the vertices $ V_6 $ and $V_7$ (of course keeping the number equal to the degree).

Moreover, creating isospectral graphs one may use not the same graphs $ \mathbf K $, but the graphs with exchanged
Steklov-equivalent building blocks -- the graphs $ \mathbf A, \mathbf B, \mathbf C, $ and $ \mathbf D$. 

It is clear that although the constructed graphs may formally not possess any symmetry, they are still symmetric
if one just looks at the involved $M$-functions for the building blocks used in the construction.   
One may treat separately the eigenvalues with the eigenfunctions supported by single blocks
(invisible eigenvalues) and the rest of the spectrum.

The presented example illustrates the enormous possibilities to construct isospectral graphs provided by our methods.
We plan to return back to this question in one of our forthcoming publications.

\section{Comparison to alternative constructons}

The presented construction determines not only isospectral metric graphs, but also discrete graphs isospectral
with respect to the normalised Laplacian -- it is enough to reduce our construction to unilateral metric graphs and use von Below
formula \cite{Below1} connecting the spectra of the normalised and differential Laplacians on these graphs. 
Let us briefly describe connections of our approach to two existing methods
to obtain isospectral graphs.

\begin{itemize}

\item {\bf Sunada approach} originally suggested in \cite{Su}, broadly applied to discrete graphs and in 
particular to metric graphs in \cite{GuSm}.

This approach uses large symmetric graphs to obtains isospectral pairs. Each graph
in the pair can be seen as a subgraph of the original large graph.
Considering the simplest isospectral pair presented in Fig. \ref{Fig82} we failed
to find any explanation via the Sunada's approach.
It might be interesting to  prove such impossibility rigorously.

\item {\bf Seidel's switching}  following \cite{Seidel,Br,BuGr}.

This construction starts from two regular graphs, which are joined together in two different ways
 (see {\it e.g.}  Fig. 6 and 7 in \cite{Br}).
They are different from the presented isospectral graphs as regularity does not play any role in
our construction.

Moreover isospectrality of normalised the Laplacians associated with discrete graphs does not
necessarily imply isospectrality of the differential Laplacian on the metric graphs obtained by
providing unit length to all edges. We have the following:

\begin{proposition}
Let $\mathbf{\Gamma}_1,\mathbf{\Gamma}_2$ be two unilateral metric graphs, and let $G_1,G_2$ be the corresponding discrete graphs. For $j=1,2$, let $L(\mathbf{\Gamma}_j)$ denote the standard differential
 Laplacian on $\mathbf{\Gamma}_j$, and let $\mathcal{L}_N(G_j)$ denote the normalised Laplacian on $G_j$. Then
\begin{equation*}
\sigma(L(\mathbf{\Gamma}_1))=\sigma(L(\mathbf{\Gamma}_2))\qquad\Leftrightarrow\qquad
\left\{
\begin{array}{l}
\sigma(\mathcal{L}_N(G_1))=\sigma(\mathcal{L}_N(G_2)), \\[3mm]
\beta_1(G_1)=\beta_1(G_2),
\end{array} \right.
\end{equation*}
where $ \sigma$ denotes the spectrum of the corresponding operator and $ \beta_1 $ denotes  the first Betti number
(the number of the independent cycles).
\end{proposition}

As a consequence of this proposition, the two graphs $ \Delta_1 $ and $ \Delta_2$ presented in Section 5 of \cite{BuGr} 
do not lead to isospectral metric graphs
 as they have different numbers of edges. They are only isospectral as discrete graphs.

\end{itemize}

\end{document}